\documentclass{article}
\usepackage{amsmath,amssymb,amsfonts,mathtools,amsthm,authblk}
\usepackage{dsfont}
\usepackage{xcolor}
\usepackage{faktor}
\usepackage{subcaption}

\usepackage{hyperref}
\hypersetup{colorlinks=true, breaklinks=true, urlcolor= blue, linkcolor= blue, citecolor=orange,linktocpage}
\usepackage[capitalize]{cleveref}

\theoremstyle{theorem}
\newtheorem{theo}{Theorem}
\newtheorem{prop}[theo]{Proposition}
\newtheorem{coro}[theo]{Corollary}
\newtheorem{lemm}[theo]{Lemma}

\newtheorem{theon}{Theorem}

\newtheorem{propn}[theon]{Proposition}

\newtheorem{coron}[theon]{Corollary}

\theoremstyle{definition}
\newtheorem{defi}{Definition}
\theoremstyle{remark}
\newtheorem*{rema}{Remark}

\DeclarePairedDelimiter{\abs}{|}{|}
\DeclarePairedDelimiter{\paren}{(}{)}
\DeclarePairedDelimiter{\subgp}{\langle}{\rangle}

\newcommand{\R}{\mathbb{R}}

\renewcommand{\H}{\mathcal{H}}

\renewcommand{\epsilon}{\varepsilon}
\newcommand{\di}{\; \mathrm{d}}
\renewcommand{\O}[1]{O \paren*{#1}}

\DeclareMathOperator{\dist}{\mathrm{dist}}

\renewcommand{\div}[1]{\paren*{\frac{#1}{2}}}
\DeclareMathOperator{\arcsinh}{arcsinh}
\DeclareMathOperator{\arccosh}{arccosh}

\DeclareMathOperator{\Vol}{Vol}
\newcommand{\id}{\mathrm{id}}

\newcommand{\Pwpo}{\mathbb{P}_g^{\mathrm{\scriptsize{WP}}}}
\newcommand{\Ewpo}{\mathbb{E}_g^{\mathrm{\scriptsize{WP}}}}
\newcommand{\Pwp}[1]{\Pwpo \left( #1 \right)}
\newcommand{\Ewp}[1]{\Ewpo \left[ #1 \right]}

\title{The tangle-free hypothesis \protect\\ on random hyperbolic surfaces}
\author[$\dagger$]{Laura Monk}
\author[$\star$]{Joe Thomas}

\affil[$\dagger$]{{\small Universit\'e de Strasbourg, CNRS, IRMA UMR 7501, F-67000 Strasbourg, France}}
\affil[$\star$]{{\small Department of Mathematics, University of Manchester, UK}}

\date{\today}

\begin{document}

\maketitle

\begin{abstract}
  This article introduces the notion of $L$-tangle-free compact hyperbolic surfaces, inspired by the identically named property for regular graphs. Random surfaces of genus $g$, picked with the Weil-Petersson probability measure, are $(a \log g)$-tangle-free for any
  $a<1$. This is almost optimal, for any surface is $(4 \log g + \O{1})$-tangled. We establish various geometric consequences of the tangle-free hypothesis at a scale~$L$, amongst which the fact that
  closed geodesics of length $< \frac L 4$ are simple, disjoint and embedded in disjoint hyperbolic cylinders of width
  $\geq \frac{L}{4}$.
\end{abstract}

\section*{Introduction}
In this article, we introduce the tangle-free hypothesis on compact (connected, oriented) hyperbolic surfaces (without
boundary), and explore some of its geometric implications, with a special emphasis on random surfaces, which we show are
almost optimally tangle-free.

This work follows several recent articles aimed at adapting results on random regular graphs in both geometry and
spectral theory to the setting of random hyperbolic surfaces -- see
\cite{mirzakhani2013,mirzakhani2019,gilmore2019,monk2020,thomas2020,magee2020} for instance. Though the initial
motivation was to provide some useful tools for spectral theory, the results and techniques developed here are purely
geometric.
Several of our results are significant improvements of useful properties of geodesics on compact hyperbolic surfaces,
allowed by the random setting: the length scale at which they apply goes from constant to logarithmic in the
genus. 

A key innovation of this article is finding an elementary geometric condition which is simultaneously easy to prove for
random surfaces, and has far-reaching consequences on their geometry (notably their geodesics) at a large scale. Similar
geometric assumptions have been made recently by Mirzakhani and Petri \cite[Proposition 4.5]{mirzakhani2019} and
Gilmore, Le Masson, Sahlsten and Thomas \cite{gilmore2019}. The use of the tangle-free hypothesis simplifies and
improves the result in \cite{gilmore2019}, and generalises one consequence of \cite[Proposition 4.5]{mirzakhani2019} to
a larger scale.

\subsection*{The tangle-free hypothesis for hyperbolic surfaces}

Let us first define what we mean by tangle-free and contrast it with existing concepts in the graph theoretic and
hyperbolic surface literature. Heuristically speaking, we shall say that a surface is tangle-free if it does not contain
embedded pairs of pants or one-holed tori with `short' boundaries. More precisely, we make the following definition.

\begin{defi}
\label{def:tanglefree}
Let $X$ be a compact hyperbolic surface and $L>0$. Then, $X$ is said to be \textit{$L$-tangle-free} if all embedded
pairs of pants and one-holed tori in $X$ have total boundary length larger than $2L$. Otherwise, $X$ is
\textit{$L$-tangled}.
\end{defi}

To be precise, we emphasise that a pair of pants and a one-holed torus are respectively surfaces of signature $(0,3)$ and $(1,1)$, and the
embedded surfaces we consider have totally geodesic boundary. The total boundary length is defined as the sum of the
length of all the boundary geodesics.  One should note that we could also have defined the notion of tangle-free using
the maximum boundary length (the length of the longest boundary geodesic) and the results of this paper would follow
through (up to changes of constants).

It may not be so clear to the reader why we call such a property \textit{tangle-free}. In order to clarify this, we
prove that, when a surface is tangled, it contains a non-simple geodesic; that is, a \textit{tangled} geodesic in the
literal sense of the word.

\begin{propn}[\cref{prop:exists_geod}]
  Any $L$-tangled surface contains a self-intersecting geodesic of length smaller than $2L + 2 \pi$.
\end{propn}

\subsection*{Tangle-free graphs}

One can motivate the study of this geometric property of surfaces through the medium of regular graphs. Indeed, the naming of
this property is inspired by a similar notion Bordenave introduced in \cite{bordenave2015} in order to prove
Friedman's theorem \cite{friedman2008} regarding the spectral gap of the Laplacian on large regular graphs (we
shall come back to this result in more detail at the end of the introduction).  A graph $G = (V,E)$ is said to be
\textit{$L$-tangle-free} if, for any vertex $v$, the ball for the graph distance $\dist_G$
\begin{equation*}
  \mathcal{B}_L(v) = \left\{ w \in V \, : \, \dist_G(v,w) \leq L \right\},
\end{equation*}
contains at most one cycle.
This definition might seem quite different to the surface definition given above, but we shall prove that balls on tangle-free
surfaces contain at most one `cycle' in the following sense.

\begin{propn}[\cref{prop:neigh_ball_cyl}]
  If a surface $X$ is $L$-tangle-free, then for any point $z \in X$, the ball
  \begin{equation*}
    \mathcal{B}_{\frac L 8}(z) = \left\{ w \in X \, : \, \dist_X(z,w) < \frac L 8 \right\}
  \end{equation*}
  is isometric to a ball in the hyperbolic plane or a hyperbolic cylinder.
\end{propn}

It is worth noting that in the original proof by Friedman \cite{friedman2008}, there is also a notion of `supercritical
tangle' in a graph, which are small subgraphs with many cycles. In a sense, pairs of pants or one-holed tori with small total
boundary lengths can be seen as analogues of these bad tangles for surfaces.

\subsection*{Admissible values of $L$}
Let us now discuss typical values that $L$ can take in \cref{def:tanglefree} both for being tangle-free and
tangled. Throughout, we shall use the notation $A = O(B)$ to indicate that there is a constant $C>0$ such that
$|A|\leq C|B|$ with $C$ independent of all other variables such as the genus.

It is clear that a surface of injectivity radius $r$ is $r$-tangle-free, for it has no closed geodesic of
length smaller than $2r$. In a deterministic setting, it is hard to say much more than this.

On the other hand, we know that a hyperbolic surface of genus $g$ admits a pants decomposition with all boundary
components smaller than the \emph{Bers constant} $\mathcal{B}_g$ -- see \cite[Chapter 5]{buser1992}. We know that
$\mathcal{B}_g \geq \sqrt{6g} - 2$ \cite[Theorem 5.1.3]{buser1992}, and the best known upper bounds on $\mathcal{B}_g$
are linear in $g$ \cite{buser1992a,parlier2014}. All surfaces of genus $g$ are
$\frac{3}{2}\mathcal{B}_g$-tangled. This bound however is rather loose, since it follows from cutting \emph{all} of the
surface into pairs of pants rather than isolating a single short pair of pants. In light of this, we in fact prove the
following, using a method based on Parlier's work \cite{parlier2014}.

\begin{propn}[\cref{prop:upper_bound}]
  Any hyperbolic surface of genus $g$ is $L$-tangled for $L = 4 \log g + O(1)$.
\end{propn}

\subsection*{Random graphs and surfaces}

How tangle-free can a \textit{typical} surface be? Can $L$ be much larger than the injectivity radius for a large class
of surfaces? An instructive method to answer these questions is to consider the setting of random
surfaces, and to find an $L$ for which most surfaces are $L$-tangle-free.

For $d$-regular graphs with $n$ vertices, sampled with the uniform probability measure $\mathbb{P}_n^{(d)}$, Bordenave
proved \cite{bordenave2015} that for any real number $0 < a < \frac 1 4$,
\begin{equation*}
  \mathbb{P}_n^{(d)}(G \text{ is } (a \log_{d-1}(n)) \text{-tangle-free})
  \underset{n \rightarrow + \infty}{\longrightarrow} 1.
\end{equation*}
This is a key ingredient in Bordenave's proof of Friedman's theorem \cite{bordenave2015}.

In this article, we will consider the Weil-Petersson probability $\mathbb{P}^{\mathrm{WP}}_g$ on the set of closed
hyperbolic surfaces of genus $g$. However, one should note that there exist other non-equivalent random surface models
such as that of Brooks and Makover~\cite{brooks2004} or Magee, Naud and Puder~\cite{magee2020}. We introduce the model
in detail in \cref{sec:random-surfaces}, and then prove that, in this setting, random surfaces are tangle-free at a scale
$\log g$ with high probability.
\begin{theon}[\cref{theo:random}]
  For any real number $0 < a < 1$,
  \begin{equation*}
    \Pwp{X \text{ is } (a \log g) \text{-tangle-free}} = 1 - O \left( \frac{(\log g)^2}{g^{1-a}} \right). 
  \end{equation*}
\end{theon}
Since any surface of genus $g$ is $(4 \log g + O(1))$-tangled, random surfaces are almost as tangle-free as
possible. The scale $\log g$ is a very large scale on a random hyperbolic surface of high genus: by work of
Mirzakhani~\cite{mirzakhani2013} the diameter of such a surface is $\leq 40 \log g$ with high probability. Mirzakhani
and Petri~\cite{mirzakhani2019} also proved that the mean value of its injectivity radius goes to a constant value
$\simeq 0.807$ as $g$ approaches infinity, hence proving that random surfaces of high genus have short closed
geodesics. These closed geodesics do not bound any pair of pants.

\subsection*{Geometric implications of the tangle-free hypothesis}

The $L$-tangle-free hypothesis has various consequences on the local geometry of the surface at a scale (roughly) $L$,
which we explore in \cref{sec:geometry}. This will be particularly interesting when $L$ is large; in the case of random
surfaces notably, where $L = a \log g$ for $a < 1$. All the results are stated for any $L$-tangle-free surface, with a
general $L$ and no other assumption, so that they can be directly applied to another setting in which a tangle-free
hypothesis is established.

First and foremost, we analyse the embedded cylinders around simple closed geodesics. In a hyperbolic surface with no
further geometric assumptions to it, the standard collar theorem \cite[Theorem 4.1.1]{buser1992} proves that the collar
of width $\arcsinh\left(\sinh\left(\ell/2\right)^{-1}\right)$ around a simple closed geodesic of length $\ell$ is an
embedded cylinder; moreover, at this width, disjoint simple closed geodesics have disjoint collars. The width of this
deterministic collar is optimal and very satisfying for small~$\ell$. For larger values of $\ell$ however, it becomes
very poor. Under the tangle-free hypothesis, we are able to obtain significant improvements to the collar theorem that
remedy this issue at larger scales.
\begin{theon}[\cref{theo:neigh_cyl}]
  On a $L$-tangle-free hyperbolic surface, the collar of width $\frac{L - \ell}{2}$ around a closed geodesic of length
  $\ell<L$ is isometric to a cylinder.
\end{theon}
This implies that we can find wide collars around geodesics of size $a \log g$, $a < 1$, on random surfaces; as a
comparison, the width of the deterministic collar around such a geodesic decreases like $g^{- \frac a 2}$. By a volume
argument, \cref{theo:neigh_cyl} is optimal up to multiplication of the width by a factor two.

The methodology to prove this result is to examine the topology of an expanding neighbourhood of the geodesic. Since the
two simplest hyperbolic subsurfaces (namely the pair of pants and one-holed torus) cannot be encountered up to a scale $\sim L$
due to the tangle-free hypothesis, the neighbourhood remains a cylinder.

An immediate consequence of this improved collar theorem is a bound on the number of intersections of a closed geodesic
of length $\ell < L$ and any other geodesic of length $\ell'$. We prove in \cref{coro: intersection of simple geodesics}
that two such geodesics intersect at most $1+\frac{\ell'}{L-\ell}$ times (and we can remove the $1$ if the two geodesics
are closed). Therefore, two closed geodesics of length $< \frac L 2$ do not intersect; \cref{prop:disjoint_collar}
furthermore states that the collars of width $\frac L 2 - \ell$ around two such geodesics are disjoint.

As well as the neighbourhood of geodesics, one can look at the geometric consequences that the tangle-free hypothesis
has on the neighbourhood of points. To this end, we explore the set of geodesic loops based at a point on the surface on
length scales up to $L$.  As has already been mentioned above in \cref{prop:neigh_ball_cyl}, which establishes a link
between our tangle-free definition and that of graphs, on an $L$-tangle-free surface, balls of radius $\frac L 8$ are
isometric to balls in either the hyperbolic plane or a hyperbolic cylinder.  There are several ways to prove this
property, some of which are similar to the proof of the improved collar theorem. In order to present different methods,
we rather deduce it from the following slightly more general result.

\begin{theon}[\cref{theo:loops}]
  If $z$ is a point on a $L$-tangle-free surface, and $\delta_z$ is the shortest geodesic loop based at $z$, then any
  other loop $\beta$ based at $z$ such that $\ell(\delta_z) + \ell(\beta) < L$ is homotopic to a power of $\delta_z$.
\end{theon}

Another consequence of \Cref{theo:loops} is \cref{coro:simple}, which states that any closed geodesic of length $<L$ on
a $L$-tangle-free surface is simple. Put together, these observations imply the following corollary.

\begin{coron}
  On a $L$-tangle-free hyperbolic surface,
  \begin{enumerate}
  \item all closed geodesics of length $< L$ are simple;
  \item all closed geodesics of length $< \frac L 2$ are pairwise disjoint;
  \item all closed geodesics of length $< \frac L 4$ are embedded in pairwise disjoint cylinders of width $\geq \frac L 4$.
  \end{enumerate}
\end{coron}
In the random case, this result is an improvement of the very useful collar theorem II \cite[Theorem 4.1.6]{buser1992},
which states that all closed geodesics of length $< 2 \arcsinh 1$ on a hyperbolic surface are simple and do not
intersect.

Short closed geodesics in random hyperbolic surfaces have been studied by Mirzakhani and Petri \cite{mirzakhani2013,
  mirzakhani2019}. One can deduce from \cite[Proposition 4.5]{mirzakhani2019} and Markov's inequality that, for any
fixed $M$,
\begin{equation*}
  1 - \Pwp{\text{all closed geodesics of length } < M \text{ are simple}} \leq \frac{C_M}{g} 
\end{equation*}
for a constant $C_M>0$, when we prove that, for any real number $0 < a < 1$, 
\begin{equation*}
  1 - \Pwp{\text{all closed geodesics of length } < a \log g \text{ are simple}} \leq C \frac{(\log g)^2}{g^{1-a}}
\end{equation*}
for a constant $C>0$.
In order to push the proof in \cite{mirzakhani2019} to a scale $\log g$, one would need to use strong properties of the
Weil-Petersson volumes and deal with technical estimates, while our approach is quite elementary in both the geometric
and probabilistic sense.

As illustrated in \cref{sec:random-surfaces}, the tools used to study random surfaces in the Weil-Petersson setting
require to reduce problems to the study of multicurves. Knowing that all closed geodesics of length $< \frac a 2 \log g$
form a multicurve can be useful to the understanding of other properties of random surfaces.

Furthermore, McShane and Parlier proved in \cite{mcshane2008} that for any $g \geq 2$,
\begin{equation*}
  \Pwp{\text{the \emph{simple} length spectrum has no multiplicity}} = 1,
\end{equation*}
where the simple length spectrum of a surface is the list of all the lengths of its simple closed geodesics.
\cref{coro:simple} then implies the following.
\begin{coro}
  For any $a \in (0,1)$, if $\mathcal{L}(X)$ denotes the length spectrum of $X$, then
  \begin{equation*}
    \Pwp{\mathcal{L}(X) \cap [0, a \log g] \text{ has no multiplicity}} = 1 - O \left(\frac{(\log g)^2}{g^{1-a}} \right).
  \end{equation*}
\end{coro}
This could be surprising since, by the work of Horowitz and Randol, for any compact hyperbolic
surface, the length spectrum has unbounded multiplicity \cite[Theorem 3.7.1]{buser1992}. However, these high multiplicities are
constructed in embedded pairs of pants, and therefore it is natural that their lengths are large for tangle-free
surfaces.

\subsection*{Motivations in spectral theory}

To conclude this introduction we will outline the connection between the geometry of hyperbolic surfaces and their
spectral theory and in particular discuss how the tangle-free hypothesis and its implications on the geometry of
surfaces on $\log g$ scales, which is a crucial scale in spectral theory, could be used to tackle some open problems in
this area. As promised, let us first return to the relation of the tangle-free hypothesis with spectral graph theory and
contrast this with that of surfaces.

\subsubsection*{Friedman's theorem}

Let $G$ be a $d$-regular graph, and $A$ be its adjacency matrix. We will call eigenvalues of $G$ the eigenvalues of the
matrix $A$. They are linked to the eigenvalues of the Laplacian $\Delta$ through the relation
$- A + d \, \mathrm{Id} = \Delta$. The value $d$ is always an eigenvalue of $G$ corresponding to constant functions, and
$-d$ is an eigenvalue if and only if $G$ is bipartite; both $d$ and $-d$ are referred to as trivial
eigenvalues. Friedman's theorem \cite{friedman2008}, first conjectured by Alon \cite{alon1986}, states that for any
$\epsilon>0$,
\begin{equation*}
  \mathbb{P}_n^{(d)} \paren*{\forall \lambda \text{ non-trivial eigenvalue of } G, \abs{\lambda} < 2 \sqrt{d-1} + \epsilon}
  \underset{n \rightarrow + \infty}{\longrightarrow} 1.
\end{equation*}
This means that large random regular graphs have an optimal spectral gap, by a result of Alon \cite{nilli1991}.

Let us compare this to what one may expect of surfaces. We will refer to the spectrum of a compact hyperbolic surface
$X$ as meaning the spectrum of the (positive) Laplace-Beltrami operator $\Delta$ on $X$. It is a non-decreasing sequence
of eigenvalues $(\lambda_n)_{n \geq 0}$, $\lambda_n \geq 0$. The value $\lambda_0=0$ is known as the trivial eigenvalue;
it is simple and the corresponding eigenfunctions are the constant functions. The equivalent surface conjecture of the
Friedman theorem was formulated by Wright \cite{wright2020}, and states that for any small enough $\epsilon > 0$,
\begin{equation*}
  \Pwp{\lambda_1 \geq \frac 1 4 - \epsilon}
  \underset{g \rightarrow + \infty}{\longrightarrow} 1.
\end{equation*}
Note that this conjecture could concern any reasonable probabilistic setting, and the remarkable properties of the
Weil-Petersson model (like Wolpert's magic formula \cite{wolpert1981} and Mirzakhani's integration formula
\cite{mirzakhani2007}) make it an excellent candidate.
Recently, Magee, Naud and Puder \cite{magee2020} have proved 
that if $X$ is a surface such that $\lambda_1(X) \geq \frac{1}{4}$ (such a surface exists \cite{jenni1984}), then for
any $\varepsilon>0$,
\begin{equation*}
  \mathbb{P}_n^{\mathrm{RC}} \left(\lambda_1(Y) \geq \frac{3}{16} - \varepsilon \right)
\underset{n \rightarrow + \infty}{\longrightarrow} 1
\end{equation*}
where $Y$ is sampled uniformly amongst the finite number of degree $n$ covers of $X$. The conjecture with $\frac 14$ instead of
$\frac{3}{16}$ is still open in this setting too.

\subsubsection*{Short cycles on graphs and surfaces}

In spectral theory, when studying large-scale limits ($n \rightarrow + \infty$ for a graph, $g \rightarrow + \infty$ for
a surface), it is important to know that the small-scale geometry of the object will not affect the spectrum.  Often, a
simple assumption to avoid this is to assume the injectivity radius to be large. 

Unfortunately, random regular graphs have an asymptotically non-zero probability of having a small injectivity radius
(see  \cite{wormald1981}). The same occurs with surfaces taken with the Weil-Petersson probability, by work of
Mirzakhani \cite{mirzakhani2013}.  As a consequence, in both cases, if we want to prove results true with probability
approaching $1$ in the large-scale limit, one needs to impose weaker and more typical geometric conditions.

For instance, Brooks and Lindenstrauss \cite{brooks2013} and Brooks and Le Masson \cite{brooks2020} studied eigenfunctions on regular graphs
of size $n \rightarrow + \infty$, under assumptions on the number of cycles up to a certain length $L$. This parameter
$L$ can always be taken to be the injectivity radius, but in the case of random graphs, it can be increased to be of
order $\log n$.  In a recent article of Gilmore, Le Masson, Sahlsten and Thomas \cite{gilmore2019}, a similar geometric
hypothesis on the number of geodesic loops shorter than a scale $L$ based at each point is made, in order to control the
$L^p$-norms of eigenfunctions of the Laplacian on hyperbolic surfaces. The authors prove it holds for random surfaces of
high genus~$g$ at a scale $L = c \log g$, but the proof provides no explicit value of the constant $c>0$.

This limitation could be seen as originating from the methodology used to study the geometry of the surfaces. In
essence, the authors prove that the loop condition is implied by a geometric condition, which is typical. This
condition however is quite complex, and both the proof of its sufficiency and typicality are rather technical, leaving the local
geometry of the random surfaces that are selected to remain quite opaque.

It follows from \cref{cor:close_Gamma} that the constant $c$ in \cite{gilmore2019} can be taken to be any value
$<\frac{1}{4}$. In turn, this improves (and makes precise) the rate of convergence of the probability for which the
$L^p$-norm estimates in \cite{gilmore2019} hold. This is rather demonstrative of the capabilities of the tangle-free geometric condition allowing for a firm grasp over $\log(g)$ scale geometries for spectral theoretic purposes. 


\subsubsection*{Benjamini-Schramm convergence}

The notion of Benjamini-Schramm convergence is another way to study spectral properties of graphs and surfaces in the
large-scale limit despite the existence of short cycles. In both cases, there is a general definition of Benjamini-Schramm
convergence of a sequence to a limiting object \cite{benjamini2001,abert2011,abert2017}, but when the limit is the
infinite $d$-regular tree (for graphs) or the hyperbolic plane (for surfaces), the definition is equivalent to a simpler
characterisation.
A sequence of hyperbolic surfaces $(X_g)_g$ converges to the hyperbolic plane if and only if
\begin{equation*}
  \forall R>0, \quad \frac{\Vol (\{ z \in X_g \, : \, \mathrm{injrad}_{X_g}(z) < R \})}{\Vol (X_g)}
  \underset{g \rightarrow + \infty}{\longrightarrow} 0,
\end{equation*}
and the definition for graphs is the same, replacing volumes by cardinalities.

Random graphs and surfaces satisfy this property for an $R$ proportional to $\log n$ and $\log g$ respectively, and this
has consequences on their eigenvalues and eigenfunctions (see Anantharaman, Le Masson \cite{anantharaman2015} and
Anantharaman \cite{anantharaman2017} for graphs, Le Masson, Sahlsten \cite{lemasson2017} and Monk \cite{monk2020} for
surfaces). The notion of Benjamini-Schramm convergence and the tangle-free hypothesis correspond to assuming the objects
have few cycles, but in different ways. The former means that the points which are the base of \emph{at least
  one} short loop are \emph{scarce} on the surface, while the latter implies that \emph{no} point has \emph{more than
  one} loop. Both approaches can be useful in different settings.

\subsection*{Outline of the paper}

The paper is organised as follows: 
\begin{itemize}
\item \Cref{sec:tangled-surfaces}: tangled surfaces have tangled geodesics.
\item \Cref{sec:random-surfaces}: random surfaces are $(a \log g)$-tangle-free for any $a<1$.
\item \Cref{sec:geometry}: geometric consequences of the tangle-free hypothesis.
\item \Cref{sec:upper_bound_L}: any surface of genus $g$ is $(4 \log g + \O{1})$-tangled.
\end{itemize}

\subsection*{Acknowledgements}

The authors would like to thank Nalini Anantharaman, Francisco Arana-Herrera, Etienne Le Masson and Tuomas Sahlsten for
valuable discussions and comments. JT is also grateful for the hospitality of the IRMA at Universit\'{e} de Strasbourg
during February 2020 during which some of this work was conducted.

\section{Tangled surfaces have tangled geodesics}
\label{sec:tangled-surfaces}

The aim in this section is to prove that being tangled implies having a tangled geodesic - that is to say a
\emph{non-simple} closed geodesic of length $\leq 2 L + \O{1}$.

\begin{prop}
  \label{prop:exists_geod}
  Let $X$ be a compact hyperbolic surface and $L >0$.  Assume that $X$ is $L$-tangled. Then, there exists a closed
  geodesic $\gamma$ in $X$ of length smaller than $2L + 2 \pi$ with one self-intersection.
\end{prop}

The geodesic we construct is what is called a \emph{figure eight}. Any non-simple geodesic on a hyperbolic surface has
length greater than $4 \arcsinh 1 \approx 3.52 \ldots$, and this result is sharp (see \cite[Theorem 4.2.2]{buser1992}).

\begin{proof}
  It suffices to prove that there is such a geodesic in any pair of pants or one-holed torus of total boundary length smaller
  than $2L$.
  \begin{figure}[h]
    \centering
    \begin{subfigure}[b]{0.5\textwidth}
      \centering
      \includegraphics[scale=0.4]{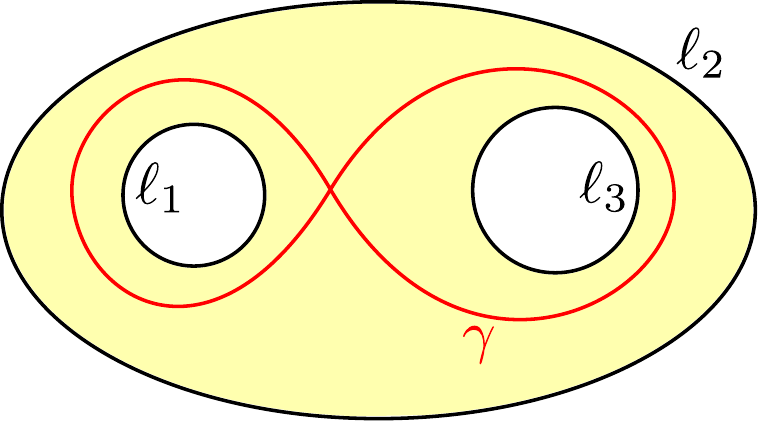}
      \caption{for a pair of pants}
      \label{fig:loops_pants}
    \end{subfigure}%
    \begin{subfigure}[b]{0.5\textwidth}
      \centering
      \includegraphics[scale=0.45]{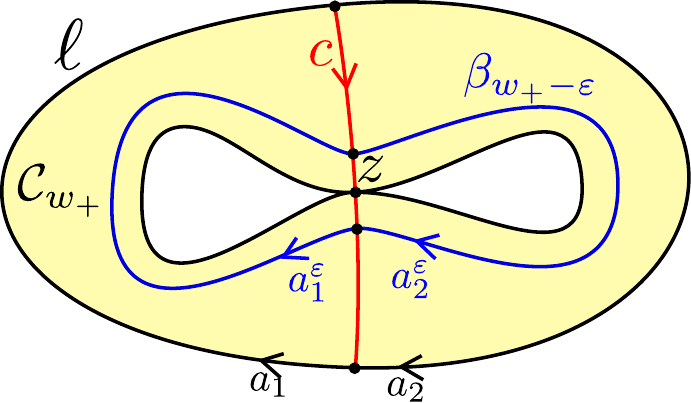}
      \caption{for a one-holed torus}
      \label{fig:loops_handle}
    \end{subfigure}%
    \caption{Construction of a short self-intersecting geodesic.}
    \label{fig:loops}
  \end{figure}
 
  Let us first consider a hyperbolic pair of pants of boundary lengths $\ell_1$, $\ell_2$, $\ell_3$, such that
  $\ell_1 + \ell_2 + \ell_3 < 2L$. We construct a closed curve with one self-intersection as represented in 
  \cref{fig:loops_pants}. By \cite[Formula 4.2.3]{buser1992}, 
  \begin{align*}
    \cosh \div{\ell(\gamma)}
    & = 2 \cosh \div{\ell_1} \cosh \div{\ell_3}
      + \cosh \div{\ell_2} \leq 3 \, e^{L}.
  \end{align*}
  Since $\cosh x \geq \frac{e^x}{2}$, we deduce that the length of $\gamma$ is smaller than $2L + 2\log 6$.
  
  We use a different proof in the one-holed torus case, because we do not have access to several small geodesics
  straight away. Let us study a one-holed torus $T$ of boundary length $\ell \leq 2L$. Let $w >0$, and $\mathcal{C}_w$
  be the $w$-neighborhood of the boundary geodesic
  \begin{equation*}
    \mathcal{C}_w = \{ z \in T \, : \, \dist(z,\partial T) < w \}.
  \end{equation*}
  By the collar theorem \cite[Theorem 4.1.1]{buser1992}, when $w$ is small enough, $\mathcal{C}_w$ is a
  half-cylinder with Fermi coordinates $(\rho,t)$, in which the hyperbolic metric is
  $\mathrm{d} s^2 = \di \rho^2 + \cosh^2 \rho \di t^2$. This isometry has to break down at some point, because the area of the
  one-holed torus is $2 \pi$, and as long as the isometry holds
  \begin{equation}
    \label{eq:area_handle_tangled}
    \Vol (\mathcal{C}_w) = \int_0^\ell \int_0^w \cosh \rho \di \rho \di t = \ell \sinh w \leq 2 \pi.
  \end{equation}
  We pick $w_+$ to be the supremum of the widths for which the isometry holds. By continuity, $w_+$ satisfies inequality
  \eqref{eq:area_handle_tangled}.

  The fact that the isometry ceases implies that there is (at least) one self-intersection point $z$ at the boundary of
  $\mathcal{C}_{w_+}$. By definition, there are two distinct geodesic segments $c_1$, $c_2$ of length $w_+$ from
  $\partial T$ to $z$. Furthermore, these segments are orthogonal to the inner boundary
  $\beta_{w_+} := \partial \mathcal{C}_{w_+} \setminus \partial T$ of the $w_+$-neighbourhood $\mathcal{C}_{w_+}$. By
  minimality of $w_+$, the two tangents of $\beta_{w_+}$ at $z$ are aligned, and therefore the two segments $c_1$, $c_2$
  connect to form a geodesic segment $c$ from $\partial T$ to itself.
  
  The regular neighbourhood of $\partial T$ and $c$ is a topological pair of pants, with three boundary components,
  $\gamma_1$, $\partial T$, $\gamma_2$. Neither $\gamma_1$ nor $\gamma_2$ is contractible because they are freely
  homotopic to geodesic bigons ($c$ and a portion of $\partial T$). Then, replacing $\gamma_1$ and $\gamma_2$ by the
  closed geodesics $\tilde \gamma_1$, $\tilde \gamma_2$ in their respective free homotopy classes yields a decomposition
  of the handle into a pair of pants of boundary components $(\tilde \gamma_1, \partial T, \tilde \gamma_2)$. Let
  $\gamma$ denote the figure-eight geodesic constucted in the pair-of-pants case, which is freely homotopic to
  $a_1 c a_2^{-1} c$, where $a_1$ and $a_2$ are the portions of $\partial T$ delimited by $c$ as represented in
  \cref{fig:loops_handle}. We shall estimate the length of $\gamma$.

  Let $\epsilon > 0$.  
  We observe that the portion $c_\epsilon$ of the geodesic segment $c$ outside of $\mathcal{C}_{w_+-\epsilon}$ is a
  geodesic segment of length $2 \epsilon$, connecting two points of $\beta_{w_+-\epsilon}$. Let $a_1^\epsilon$,
  $a_2^\epsilon$ be the two portions of $\beta_{w_+-\epsilon}$ delimited by $c_\epsilon$. Then, the loop
  $a_1^\epsilon c_\epsilon (a_2^\epsilon)^{-1} c_\epsilon$ is freely homotopic to $a_1 c a_2^{-1} c$ and hence
  $\gamma$. Its length is equal to
  \begin{equation*}
    \ell(\beta_{w_+-\epsilon}) + 4 \epsilon
    = \ell \cosh (w_+ - \epsilon) + 4 \epsilon
    \underset{\epsilon \rightarrow 0}{\longrightarrow} \ell \cosh (w_+).
  \end{equation*}
  By minimality of the geodesic representative in a free homotopy class,
  \begin{equation*}
    \ell(\gamma) \leq \ell \cosh(w_+) = \ell \sqrt{1 + \sinh^2(w_+)} \leq \ell \sqrt{1 + \frac{4 \pi^2}{\ell^2}} \leq 2 L + 2 \pi
  \end{equation*}
  by equation \eqref{eq:area_handle_tangled}, which allows us to conclude.
\end{proof}

\section{Random surfaces are $(a \log g)$-tangle-free}
\label{sec:random-surfaces}

In this section, we will show that, for any $0<a<1$, typical surfaces of genus $g$ are $(a\log g)$-tangle-free. By
typical we mean in the probabilistic sense for the Weil-Petersson model of random surfaces. To be precise we shall
introduce this model briefly here, a more thorough overview can be found in \cite{imayoshi1992} or \cite{wright2020}.

\subsection{Teichm\"uller and moduli spaces}
\label{sec:teichm-moduli}

For integers $g,n$ such that $2g-2+n > 0$, fix a connected and oriented smooth surface $S_{g,n}$ of genus $g$ and with $n$ numbered
boundary components. Let us also fix a length vector $\ell = (\ell_1,\ldots, \ell_n) \in \R^n_{> 0}$. Define the
\textit{Teichm\"{u}ller space} $\mathcal{T}_{g,n}(\ell)$ by
\begin{align*}
  \mathcal{T}_{g,n}(\ell)
  = \left\{(X,f): \parbox{8cm}{$f:S_{g,n}\to X$ diffeomorphism \\
  $X$ hyperbolic surface \\
  $i$-th boundary component of length $\ell_i$ for $1 \leq i \leq n$}\right\}
  \Big/ \sim,
\end{align*}
where $\sim$ is the equivalence relation $(X_1,f_1)\sim (X_2,f_2)$ if and only if there exists an isometry
$h: X_1 \to X_2$ such that $f_2\circ h\circ f_1^{-1}:S_{g,n}\to S_{g,n}$ is isotopic to the identity.

The elements of $\mathcal{T}_{g,n}(\ell)$ are surfaces with a \emph{marking}. Many surfaces are isometric, but have a
different marking. If one wants to pick a random surface, it is more natural to take it in the \textit{moduli space}
\begin{equation*}
  \mathcal{M}_{g,n}(\ell)
  = \left\{ \; \parbox{6.5cm}{hyperbolic surfaces of genus $g$ \\
      with $n$ boundary components \\
      $i$-th component of length $\ell_i$ for $1 \leq i \leq n$}\right\}
  \big/ \text{\small \{isometry\}}
\end{equation*}
where the quotient is over the set of isometries that preserve the $i$-th component setwise, for all
$i \in \{1, \ldots, n\}$.  The moduli space can be obtained as a quotient of the Teichm\"{u}ller space by the action of
the mapping class group
\begin{align*}
  \mathcal{M}_{g,n}(\ell) = \mathcal{T}_{g,n}(\ell)/\mathrm{MCG}(S_{g,n}).
\end{align*}
We recall that $\mathrm{MCG}(S_{g,n})$ is the group of orientation preserving diffeomorphisms of $S_{g,n}$ that setwise
preserve the boundary components of the surface, up to isotopy, and it acts on the Teichm\"uller space by precomposition
of the marking.

In the case when $n = 0$ (and the surface is compact, with no boundary), we will suppress the mention of $n$ (and the
empty vector $\ell$), and write $S_g$, $\mathcal{M}_{g}$, $\mathcal{T}_g$.

\subsection{The Weil-Petersson probability}
\label{sec:weil-peterss}

The Teichm\"uller space $\mathcal{T}_{g,n}(\ell)$ possesses a natural symplectic structure, the \emph{Weil-Petersson form}
$\omega^{\mathrm{WP}}_{g,n,\ell}$, which is invariant under the action of the mapping class group and therefore descends to
the moduli space.

The symplectic form induces a volume form
$\mathrm{dVol}^{\mathrm{WP}}_{g,n,\ell} = \frac{1}{N!} (\omega^{\mathrm{WP}}_{g,n,\ell})^{\wedge N}$ for $N = 3g-3+n$, called the
\textit{Weil-Petersson volume form}.  The volume of the moduli space is a finite quantity
\begin{equation*}
  V_{g,n}(\ell) :=\mathrm{Vol}^{\mathrm{WP}}_{g,n,\ell}(\mathcal{M}_{g,n}(\ell)).\label{eq:1}
\end{equation*}
When $n=0$ (and the surface is compact, with no
boundary), we write $\mathrm{Vol}_g^{\mathrm{WP}}$ and $V_g$ to simplify notations. We will see in the next section why
we need to introduce these volumes for surfaces with boundary components, even when we only want to study boundary-free compact
surfaces.

We can normalise $\mathrm{Vol}^{\mathrm{WP}}_{g}$ and obtain the \textit{Weil-Petersson probability measure}
$\mathbb{P}_{g}^{\mathrm{WP}} = \frac{1}{V_{g}} \mathrm{Vol}^{\mathrm{WP}}_{g}$ on the moduli space $\mathcal{M}_{g}$.
The Weil-Petersson form can be expressed in Fenchel-Nielsen coordinates thanks to Wolpert's theorem
\cite{wolpert1981}. This geometric expression has deep consequences, and is what ultimately allows for explicit
computations in this model.

\subsection{Mirzakhani's integration formula}
\label{sec:mirz-integr-form}

In this subsection, we explain how Mirzakhani's integration formula \cite{mirzakhani2007} can be used to compute
expectations of a certain class of functions known as geometric functions. Knowing how to compute expectations then
allows one to estimate the probability of certain events by, for instance, using Markov's inequality
$\mathbb{P}(|X|>a) \leq \frac{1}{a} \; \mathbb{E}(|X|)$.

\begin{defi}
  A \emph{geometric function} is a function $\mathcal{M}_g \rightarrow \R$ that can be written as:
  \begin{align*}
    F^\Gamma(X) = \sum_{(\gamma_1,\ldots,\gamma_k) \in \mathcal{O}(\Gamma)} F(\ell_X(\gamma_1),\ldots,\ell_X(\gamma_k)),
  \end{align*}
  where:
  \begin{itemize}
  \item $F : \R_{\geq 0}^k \rightarrow \R$ is a positive measurable function
  \item $\Gamma$ is a multi-curve on $S_g$, and $\mathcal{O}(\Gamma)$ is the orbit of $\Gamma$ under the action by
    the mapping class group $\mathrm{MCG}(S_g)$
  \item for a closed curve $\gamma$ on $S_g$ and $(X,f) \in \mathcal{T}_{g}$, $\ell_X(\gamma)$ is the length of the
    unique closed geodesic freely homotopic to the image of $\gamma$ on $X$ under the marking map $f$.
  \end{itemize}
\end{defi}
Though a fixed term of the sum in the previous definition only really makes sense for an element of the Teichm\"uller
space, the summation over the orbit makes it invariant under the action of the mapping class group, and hence a
well-defined function on the moduli space $\mathcal{M}_g$.

The following result is an expression of the integral of any geometric function as an integral over
$\mathbb{R}_{\geq 0}^k$. In order to write the formula, we must understand the surface resulting in cutting $S_g$
by the curves in $\Gamma$. For this, we observe that the cut surface $S_g\setminus\Gamma$ can be written as the
disjoint union $\bigsqcup_{i=1}^q S_{g_i,n_i}$ of its connected pieces.

The $k$ curves of $\Gamma$ form $2k$ boundary components of the cut surface.  If the multi-curve $\Gamma$ had lengths
$\ell \in \R_{\geq 0}^k$ on $X$, then these lengths become the boundary lengths of the surface $X$ cut along $\Gamma$. Each
component $S_{g_i,n_i}$ therefore has a length vector $\ell^{(i)} \in \R_{\geq 0}^{n_i}$.  We then define
\begin{align*}
  V_g(\Gamma, \ell) := \prod_{i=1}^q V_{g_i,n_i}(\ell^{(i)}).
\end{align*}
Mirzakhani's integration formula can then be formulated as follows.

\begin{theo}[\cite{mirzakhani2007}]
\label{thm:mirzakhaniintegral}
Given a multi-curve $\Gamma$ and a function $F:\mathbb{R}^k_{\geq 0}\to\mathbb{R}$ there exists a constant
$0 < C_\Gamma \leq 1$ dependent only on $\Gamma$ for which
\begin{align*}
  \int_{\mathcal{M}_g} F^\Gamma(X) \, \mathrm{dVol}_{g}^{\mathrm{WP}}(X)
  = C_\Gamma \int_{\mathbb{R}^k_{\geq 0}} F(x) \, V_g(\Gamma, \ell)
  \, \ell_1\cdots \ell_k
  \, \mathrm{d} \ell_1 \cdots \mathrm{d} \ell_k.
\end{align*}
\end{theo}

\subsection{Volume estimates}
\label{sec:volume-estimates}

The previous formula indicates that in order to estimate expectations, we need to understand the
asymptotic behaviour of Weil-Petersson volumes. 
In our proof, we will only use a handful of them, grouped in the following Lemma.

\begin{lemm}[Lemmas 3.2 and 3.3 \cite{mirzakhani2013}]
  \label{lemm: volumebounds}
  Given $g,n\geq 0$ such that $2g-2+n>0$,
  \begin{enumerate}
  \item $\ell_1 \ldots \ell_n V_{g,n}(\ell_1,\ldots,\ell_n) \leq 2^n \prod_{i=1}^n \sinh \paren*{\frac{\ell_i}{2}} V_{g,n}$,
  \item $V_{g,n+2} \leq V_{g+1,n}$,
  \item there exists a constant $C$ independent of $g$ and $n$ such that
    \begin{align*}
      V_{g,n} \leq C\frac{V_{g,n+1}}{2g-2+n},
    \end{align*}
  \item there exists a constant $C_{n}$ independent of $g$ such that for any integers $n_1, n_2$ satisfying
    $n_1+n_2=n$, 
    \begin{equation*}
      \sum_{g_1+g_2=g} V_{g_1,n_1+1} V_{g_2,n_2+1} \leq C_{n} \frac{V_{g,n}}{g} \cdot
    \end{equation*}
  \end{enumerate}
\end{lemm}

\subsection{Probabilistic result}
\label{sec:probabilistic-result}

We can now state and prove our probabilistic result.

\begin{theo}
  \label{theo:random}
  For any real number $0 < a < 1$,
  \begin{equation*}
    \Pwp{X \text{ is } (a \log g) \text{-tangle-free}} = 1 - O \left( \frac{(\log g)^2}{g^{1-a}} \right). 
  \end{equation*}
\end{theo}

%


\begin{proof}
  Let us list all the topological types of embedded one-holed tori or pair of pants in a genus $g$ surface (see
  \ref{fig:topological_ways}):
  \begin{enumerate}
  \item[(i)]  a curve separating a one-holed torus;
  \item[(ii)] three curves cutting $S_g$ into a pair of pants and a component $S_{g-2,3}$;
  \item[(iii)] three curves cutting $S_g$ into a pair of pants and  two components $S_{g_1,1}$ and $S_{g_2,2}$ such that $g_1+g_2=g-1$;
  \item[(iv)] three curves cutting $S_g$ into a pair of pants and three connected components $S_{g_1,1}$, $S_{g_2,1}$ and
    $S_{g_3,1}$ with $1 \leq g_1 \leq g_2 \leq g_3$ and $g_1+g_2+g_3 = g$.
  \end{enumerate}
    \begin{figure}[h!]
    \centering
    \includegraphics[scale=0.5]{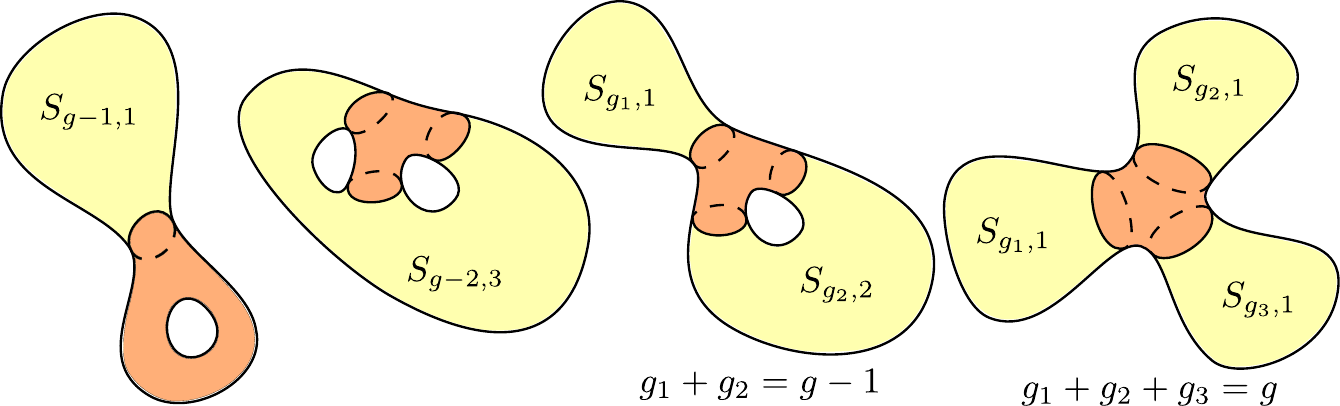}
    \caption{The different topological ways to embed a one-holed torus or pair of pants in a surface of genus $g$.}
    \label{fig:topological_ways}
  \end{figure}

  For any topological situation, we will consider a multicurve $\alpha$ on the base surface $S_g$ realising the
  topological configuration and study the counting function
  \begin{equation*}
    N^\alpha_L(X) = \# \{ \beta \in \mathcal{O}(\alpha) \; : \; \ell_X(\beta) \leq 2L \},
  \end{equation*}
  where the length of a multi-curve is defined as the sum of its components.
  Then, the probability of finding a component in the topological situation $\alpha$ of total boundary length $\leq 2L$
  can be bounded by Markov's inequality:
  \begin{equation*}
    \Pwp{N^\alpha_L(X) \geq 1} \leq \Ewp{N^\alpha_L(X)}.
  \end{equation*}
  We observe that $N^\alpha_L(X)$ is a geometric function, and its expectation can therefore be computed using
  Mirzakhani's integration formula \eqref{thm:mirzakhaniintegral}. This reduces the problem to estimating integrals with
  Weil-Petersson volumes, which we will now detail.

  In case (i), the integral that appears is
  \begin{equation*}
    \int_0^{2L} V_{1,1}(\ell)V_{g-1,1}(\ell) \, \ell\di\ell.
  \end{equation*}
  From \cite{naatanen1998}, it is known that $V_{1,1}(\ell) = \frac{\ell^2}{24} + \frac{\pi^2}{6}$.
  Moreover, by \cref{lemm: volumebounds},
  \begin{align*}
    \ell V_{g-1,1}(\ell) \leq 2e^{\frac{\ell}{2}}V_{g-1,1}.
  \end{align*}
  It follows that the probability is smaller than
  \begin{equation*}
    \frac{V_{g-1,1}}{V_g}\int_0^{2L} 2\left( \frac{\ell^2}{24} + \frac{\pi^2}{6} \right) e^{\frac{\ell}{2}}\di\ell
    = O\left(\frac{V_{g-1,1}}{V_g}L^2e^L\right)
    = O \left( \frac{(\log g)^2}{g^{1-a}} \right)
  \end{equation*}
  where the last bound is deduced from \cref{lemm: volumebounds} parts (2) and (3) and taking $L = a \log g$.

  In case (ii), the integral that appears is
  \begin{equation*}
    \frac{1}{V_g} \iiint_{0 \leq \ell_1 + \ell_2 + \ell_3 \leq 2L} V_{0,3}(\ell_1,\ell_2,\ell_3)
    V_{g-2,3}(\ell_1,\ell_2,\ell_3) \, \ell_1 \ell_2 \ell_3 \di \ell_1 \di \ell_2 \di \ell_3.
  \end{equation*}
  Due to the fact that $V_{0,3}(\ell_1,\ell_2,\ell_3) = 1$ and by \cref{lemm: volumebounds}(1), we need to estimate
  \begin{equation*}
    \frac{V_{g-2,3}}{V_g} \iiint_{0 \leq \ell_1 + \ell_2 + \ell_3 \leq 2L}
    \exp \paren*{\frac{\ell_1 + \ell_2 + \ell_3}{2}} \di \ell_1 \di \ell_2 \di \ell_3
    = O \left( \frac{(\log g)^2}{g^{1-a}} \right)
  \end{equation*}
  by \cref{lemm: volumebounds} (2-3).

  Let us now bound the sum of all the topological situations of case (iii). By the same manipulations, we obtain that
  the probability is
  \begin{equation*}
    O \left( \frac{L^2 e^L}{V_g} \sum_{g_1 + g_2 = g-1} V_{g_1,1} V_{g_2,2} \right)
    = O \left( \frac{(\log g)^2}{g^{1-a}} \frac{ V_{g-1,1}}{ V_g} \right)
    = O \left( \frac{(\log g)^2}{g^{2-a}} \right)
  \end{equation*}
  by \cref{lemm: volumebounds}(4) and then \cref{lemm: volumebounds}(2-3).

  Finally, in the last case we have to estimate
  \begin{align*}
    & \sum_{\substack{g_1 + g_2 + g_3 = g \\ 1 \leq g_1 \leq g_2 \leq g_3}} V_{g_1,1} V_{g_2,1}V_{g_3,1} \\
    & = \sum_{g_1 = 1}^{\lfloor \frac{g-2}{3}\rfloor} V_{g_1,1}  \sum_{g_2 + g_3 = g-g_1} V_{g_2,1}V_{g_3,1}
    \leq C_{0} \sum_{g_1 = 1}^{\lfloor \frac{g-2}{3}\rfloor} \frac{V_{g_1,1} V_{g-g_1,0}}{g-g_1}
  \end{align*}
  where $C_{0}$ is the constant from \cref{lemm: volumebounds}(4). We observe that $g-g_1 \geq \frac{2}{3} g$ and use
  \cref{lemm: volumebounds}(3) to conclude that the probability is 
  \begin{equation*}
    O \left( \frac{(\log g)^2}{V_gg^{2-a}} \sum_{g_1=1}^{\lfloor \frac{g-2}{3}\rfloor} V_{g_1,1}V_{g-g_1,1} \right)
    = O \left( \frac{(\log g)^2}{g^{3-a}} \right)
  \end{equation*}
  by \cref{lemm: volumebounds}(4).  
 \end{proof}

 \begin{rema}
   In the cases (i), (iii) and (iv), there is a separating geodesic of length $\leq 2a \log g$. Therefore, we could
   have bounded these probabilities by the probability of having a separating geodesic of length $\leq 2a \log g$, which
   has been estimated by Mirzakhani in \cite[Theorem 4.4]{mirzakhani2013}. This approach yields the same end result, but
   the authors decided to detail the four cases for the sake of self-containment. Furthermore, this more detailed study
   allows us to see that the most likely cases are cases (i.) and (ii.), and therefore we expect the first length at
   which the surface is tangled to be obtained by one of these two topological situations. 
\end{rema}

\section{Geometry of tangle-free surfaces}
\label{sec:geometry}

The aim of this section is to provide information about geodesics and neighbourhoods of points on tangle-free surfaces. The
results will be expressed in terms of an arbitrary $L$-tangle-free surface $X$, but can also been seen as result that are true
with high probability for $L = a \log g$, $a < 1$ due to \cref{theo:random}.

\subsection{An improved collar theorem}
\label{sec:collars}

\begin{theo}
  \label{theo:neigh_cyl}
  Let $L>0$, and $X$ be a $L$-tangle-free hyperbolic surface. Let $\gamma$ be a simple closed geodesic of length
  $\ell < L$. Then, for $w := \frac{L - \ell}{2}$, the neighbourhood
  \begin{equation*}
    \mathcal{C}_w(\gamma) = \{ z \in X \, : \, \dist(z,\gamma) < w \}
  \end{equation*}
  is isometric to a cylinder.
\end{theo}

The collar theorem \cite{buser1992} is a similar result, with the width $\arcsinh \paren*{\sinh \paren*{\ell/2}^{-1}}$.

We recall that, in the random case, for $a<1$, with high probability, we can take $L = a \log g$. This result therefore
is a significant improvement for geodesics of length $b \log g$, $0 < b < a$. We obtain a collar of width
$w = \frac{a-b}{2} \log g$, which is expanding with the genus, as opposed to the deterministic collar, of width
$\simeq g^{-\frac b 2}$.

For very short geodesics, the width of this new collar is $\simeq \frac{a}{2} \log g$. It might seem less good than the
deterministic collar, which is of width $\simeq - \log (\ell)$. However, by Theorem 4.2 in \cite{mirzakhani2013}, the
injectivity radius of a random surface is greater than $g^{-\frac a 2}$ with probability $1 - O(g^{-a})$. Under this
additional probabilistic assumption, the two collars are of similar sizes.

\begin{proof}
  For small enough $w$, the neighbourhood $\mathcal{C}_w(\gamma)$ is a cylinder, with two boundary components
  $\gamma^\pm_w$.  Let us assume that, for a certain $w$, the topology of the neighbourhood changes. There are two ways
  for this to happen (and both can happen simultaneously) -- see \cref{fig:cylinder_reg_neig}.
  \begin{enumerate}
  \item[(A)] One boundary component, $\gamma_w^+$ or $\gamma_w^-$, self-intersects.
  \item[(B)] The two boundary components $\gamma^+_w$ and $\gamma^-_w$ intersect one another. 
  \end{enumerate}

  \begin{figure}[h]
    \centering
    \begin{subfigure}[b]{0.5\textwidth}
      \centering
      \includegraphics[scale=0.4]{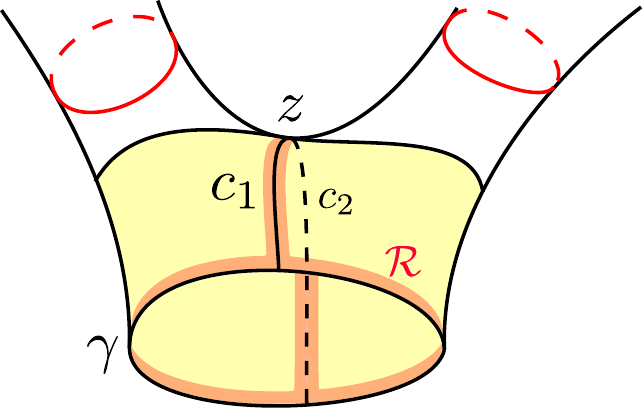}
      \caption{one side self-intersects}
      \label{fig:cylinder_reg_neig_1}
    \end{subfigure}%
    \begin{subfigure}[b]{0.5\textwidth}
      \centering
      \includegraphics[scale=0.4]{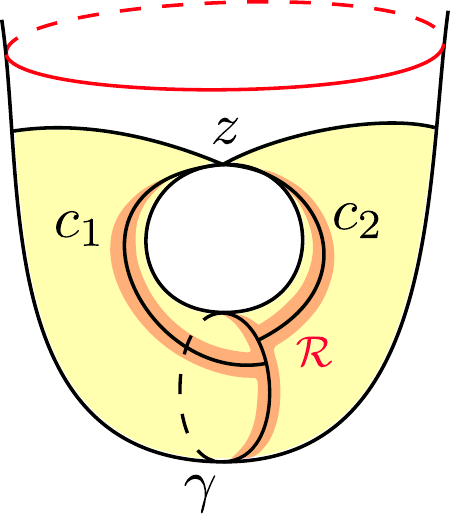}
      \caption{the two sides intersect one another}
      \label{fig:cylinder_reg_neig_2}
    \end{subfigure}%
    \caption{Illustration of the ways the isometry breaks down when expanding a cylinder around the geodesic $\gamma$.}
    \label{fig:cylinder_reg_neig}
  \end{figure}

  In both cases, let $z \in X$ denote one intersection point. Since the distance between $z$ and $\gamma$ is $w$, there
  are two distinct geodesic arcs $c_1$, $c_2$ of length $w$, going from $z$ to points of $\gamma$, and intersecting
  $\gamma$ perpendicularly. Both $c_1$ and $c_2$ are orthogonal to the boundaries of the cylinder and the two boundaries are
  tangent to one another by minimality of the width $w$. As a consequence, the curve $c = c_1^{-1}c_2$ is a geodesic arc.

  The regular neighbourhood of the curves $\gamma$ and $c$ has Euler characteristic $-1$. There are two possible
  topologies for this neighbourhood.
  \begin{itemize}
  \item If it is a pair of pants, then it has three boundary components. Neither of them is contractible on the surface
    $X$. Indeed, one component is freely homotopic to $\gamma$, and the two others to $c$ and a portion of $\gamma$,
    which are geodesic bigons. Therefore, when we replace the boundary components of the regular neighbourhood by the
    closed geodesic in their free homotopy classes, we obtain a pair of pants or a one-holed torus (if two of the boundary
    components are freely homotopic to one another), of total boundary length smaller than
    $2 \ell + 4 w$.
  \item Otherwise, it is a one-holed torus. Its boundary component is not contractible, because there is no hyperbolic surface of
    signature $(1,0)$. Therefore, the closed geodesic in its free homotopy class separates a one-holed torus with boundary length
    smaller than $2 \ell + 4 w$ from $X$.
  \end{itemize}
  In both cases, by the tangle-free hypothesis, $2L < 2\ell + 4 w$, which allows us to conclude. 
\end{proof}

\begin{rema}
  Let $\mathcal{A}_g \subset \mathcal{M}_g$ be the event ``the surface has a simple closed geodesic of length between
  $1$ and $2$''.  By work of Mirzakhani and Petri \cite{mirzakhani2019},
  \begin{equation*}
    \Pwp{\mathcal{A}_g} \underset{g \rightarrow + \infty}{\longrightarrow} 1- \exp\left(-\int_1^2 \frac{e^t + e^{-t} -2}{2t} \, \mathrm{d} t\right) > 0,
  \end{equation*}
  so this event has asymptotically non-zero probability.

  Let $X$ be an element of $\mathcal{A}_g$ which is also $(a \log g)$-tangle-free, and let $\gamma$ be a closed geodesic
  on $X$ of length $\ell \in [1,2]$. Then, the collar $\mathcal{C}_w(\gamma)$ given by \cref{theo:neigh_cyl} has
  volume
  \begin{equation*}
    \Vol (\mathcal{C}_w(\gamma)) = 2 \ell \sinh w \geq 2 \sinh \paren*{\frac{a}{2} \log g - 1} \sim g^{\frac a 2}
    \quad \text{as } g \rightarrow + \infty.
  \end{equation*}
  However, $\Vol (\mathcal{C}_w(\gamma)) \leq \Vol X = 2 \pi (2g-2)$.  This leads to a contradiction for $g$ approaching
  $+ \infty$ as soon as $a>2$. Hence, for large $g$, the elements of $\mathcal{A}_g$ are $(a \log g)$-tangled for $a>2$:
  \begin{equation*}
    \limsup_{g \rightarrow + \infty} \Pwp{X \text{ is (}a \log g \text{)-tangled}}
    \geq \lim_{g \rightarrow + \infty} \Pwp{\mathcal{A}_g} > 0.
  \end{equation*}
  Therefore, for all $a>2$, random surfaces do \emph{not} have high probability of being $(a \log g)$-tangle-free.

  By taking $a$ close to but larger than $1$, this same line of reasoning and the fact that we know surfaces to be
  $(a \log g)$-tangle-free with high probability implies that the improved collar cannot be much larger than $L -
  \ell$. As a consequence, our result is optimal up to multiplication by $2$.
\end{rema}

\subsection{Number of intersections of geodesics}
\label{sec:numb-inters-geod}

A consequence of this improved collar theorem is a bound on the number of intersections of a short closed geodesic with
any other geodesic.

\begin{coro}
  \label{coro: intersection of simple geodesics}
  Let $L >0$, and $X$ be a $L$-tangle-free hyperbolic surface.
  
  Let $\gamma$ be a simple closed geodesic of length $< L$ on $X$. Then, for any geodesic~$\gamma'$ transverse to
  $\gamma$, the number of intersections $i(\gamma,\gamma')$ between $\gamma$ and $\gamma'$ satisfies
  \begin{equation*}
    i(\gamma, \gamma') \leq \frac{\ell(\gamma')}{L-\ell(\gamma)} + 1.
  \end{equation*}
  In the case where $\gamma'$ is also closed, then
  \begin{equation*}
    i(\gamma, \gamma') \leq \frac{\ell(\gamma')}{L-\ell(\gamma)} \cdot
  \end{equation*}
  In  particular, if $\ell(\gamma) + \ell(\gamma') < L$, then $\gamma$ and $\gamma'$ do not intersect.
\end{coro}

\begin{proof}
  By \cref{theo:neigh_cyl}, $\gamma$ is embedded in an open cylinder $\mathcal{C}$ of width
  $w = \frac{L - \ell(\gamma)}{2}$.

  Let us parametrize the geodesic $\gamma' : [0,1] \rightarrow X$. The set of times when $\gamma'$ visits the cylinder
  can be decomposed as
  \begin{equation*}
    \bigsqcup_{i=1}^k (t_i^-, t_i^+), \qquad 0 \leq t_1^- < t_1^+ \leq \ldots \leq t_k^- < t_k^+ \leq 1,
  \end{equation*}
  as respresented in \cref{fig:bound_number_int}.
  \begin{figure}[h]
    \centering
    \includegraphics[scale=0.35]{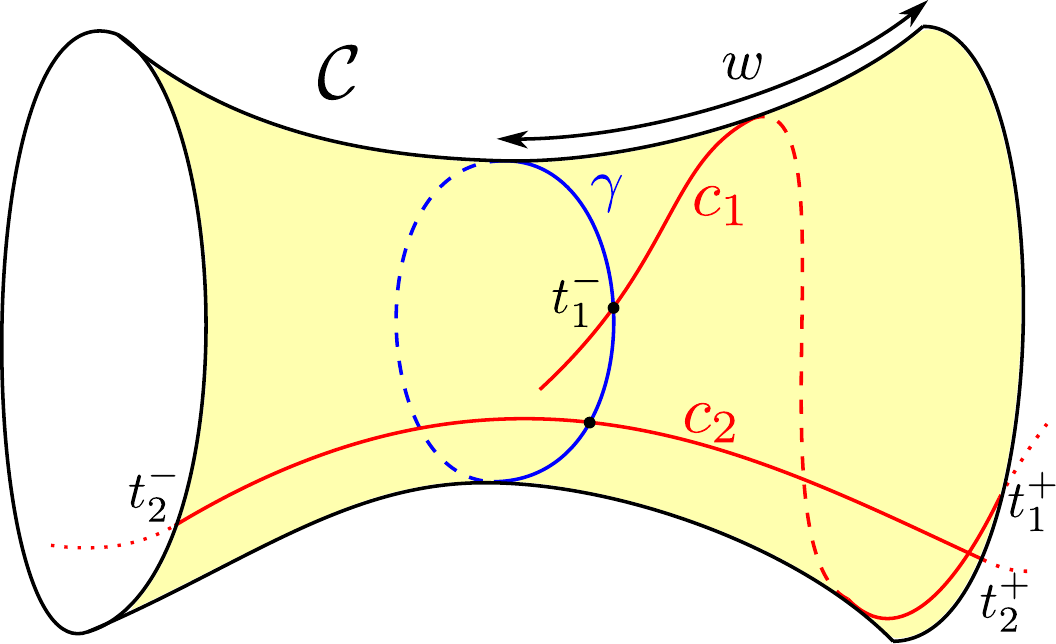}
    \caption{Illustration of the proof of \cref{coro: intersection of simple geodesics}.}
    \label{fig:bound_number_int}
  \end{figure}
  The restriction $c_i$ of $\gamma'$ between $t_i^-$ and $t_i^+$ is a geodesic in the cylinder $\mathcal{C}$,
  transverse to the central geodesic $\gamma$. Therefore, if $c_i$ intersects $\gamma$, then it does at most once. Let
  $I \subset \{ 1, \ldots, k \}$ be the set of $i$ such that $c_i$ intersect $\gamma$. We have that
  $i(\gamma, \gamma') = \# I \leq k$.

  We assume that $\#I \geq 2$ (otherwise their is nothing to prove).  Any geodesic intersecting the central geodesic
  transversally travels through the entire cylinder, and is therefore of length greater than $2 w$. As a consequence,
  for any $i \in I$ different from $1$ and $k$, $\ell(c_i) \geq 2w$. Also, if $i = 1$ or $k$ belongs in $I$, then
  $\ell(c_i) \geq w$. This leads to our claim, because
  \begin{equation*}
    (i(\gamma, \gamma') - 1) (L - \ell(\gamma)) = (\#I-1) \cdot 2w \leq \sum_{i \in I} \ell(c_i) \leq \ell(\gamma').
  \end{equation*}
  
  The case when the curve $\gamma'$ is closed can be obtained observing that, in this case, $\ell(c_1)$ and $\ell(c_k)$
  also are greater than $2w$ (when $1$ or $k$ belongs in $I$).
\end{proof}

Like the collars from the usual collar theorem, the collars of two small enough distinct geodesics are disjoint.

\begin{prop}
  \label{prop:disjoint_collar}
  Let $L>0$, and $X$ be a $L$-tangle-free hyperbolic surface.  Let $\gamma$, $\gamma'$ be two distinct simple closed
  geodesics such that $\ell(\gamma) + \ell(\gamma')< L$.  Then, the distance between $\gamma$ and $\gamma'$ is greater
  than $L - \ell(\gamma) - \ell(\gamma')$.

  In particular, if $\ell(\gamma), \ell(\gamma') < \frac L 2$, then the collars of width $\frac L 2 - \ell(\gamma)$
  around $\gamma$ and $\frac L 2 - \ell(\gamma')$ around $\gamma'$ are two disjoint embedded cylinders.
\end{prop}

\begin{figure}[h]
  \centering
  \includegraphics[scale=0.4]{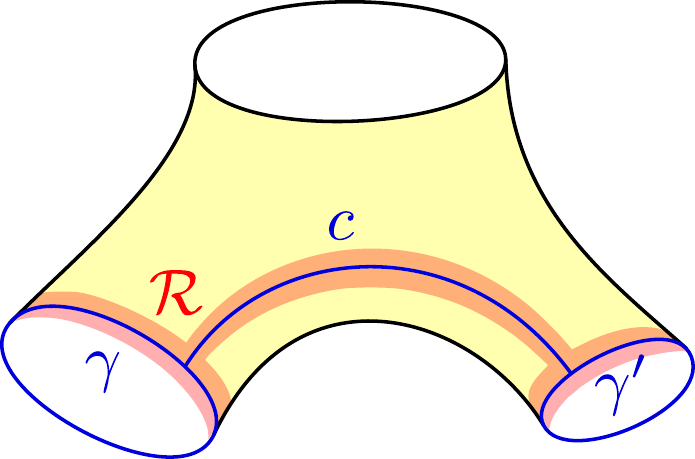}
  \caption{Illustration of the proof of \cref{prop:disjoint_collar}}
  \label{fig:disjoint_collar}
\end{figure}

\begin{proof}
  We already know, owed to \cref{coro: intersection of simple geodesics}, that $\gamma$ and $\gamma'$ do not
  intersect. Let~$c$ be a length-minimising curve with one endpoint on $\gamma$ and the other on $\gamma'$ (see
  \cref{fig:disjoint_collar}). Then, by minimality, $c$ is simple and only intersects $\gamma$ and $\gamma'$ at is
  endpoints. The regular neighbourhood $\mathcal{R}$ of $\gamma$, $\gamma'$ and $c$ is a topological pair of pants of
  total boundary length less than $2 (\ell(\gamma) + \ell(\gamma') + \ell(c))$.  Since $\gamma$ and $\gamma'$ are
  non-contractible and not freely homotopic to one another, the third boundary component is not contractible and
  $\mathcal{R}$ corresponds to an embedded pair of pants or one-holed torus on $X$.  By the tangle-free hypothesis,
  $\ell(\gamma) + \ell(\gamma') + \ell(c) \geq L$, and therefore the distance between $\gamma$ and $\gamma'$ is greater
  than $L - \ell(\gamma) - \ell(\gamma')$. This implies our claim.
\end{proof}

\subsection{Short loops based at a point}
\label{sec:short-loops-based}

Let us now study short loops based at a point on a tangle-free surface. 

 \begin{theo}
 \label{theo:loops}
 Let $L>0$, and $X$ be a $L$-tangle-free hyperbolic surface.  Let $z \in X$, and let $\delta_z$ be the shortest geodesic
 loop based at $z$.
 
 Let $\beta$ be a (non necessarily geodesic) loop based at $z$, such that
 $\ell(\beta) + \ell(\delta_z) < L$.
 Then $\beta$ is homotopic with fixed endpoints to a power of $\delta_z$. 
\end{theo}

The result is empty if the injectivity radius of the point $z$ is greater than $\frac L 2$.  The ``shortest geodesic
loop'' $\delta_z$ is not necessarily unique. It will be as soon as the injectivity radius at $z$ is smaller than
$\frac L 4$. More precisely, we directly deduce from \cref{theo:loops} the following corollary, which was used in
\cite{gilmore2019} for random surfaces (with a length $L = a \log g$, but the value of $a$ was not explicit). Note the similarity of this result to the classical Margulis lemma \cite{ratcliffe2019}. In particular, we obtain an explicit constant for the Margulis lemma in the case of tangle-free surfaces in the same way that the classical collar theorem provides.
\begin{coro}
  \label{cor:close_Gamma}
  Let $L>0$, and $X = \faktor{\H}{\Gamma}$ be an $L$-tangle-free hyperbolic surface. Then, for any $z \in \H$, the set
  $\{ T \in \Gamma \; : \; \dist_{\H} (z, T \cdot z) < \frac L 2 \}$ is:
  \begin{itemize}
  \item reduced to the identity element (when the injectivity radius at $z$ is $\geq \frac L 4$),
  \item or included in the subgroup $\subgp{T_0}$ generated by the element $T_0 \in \Gamma$ corresponding to
    the shortest geodesic loop through $z$.
  \end{itemize}
\end{coro}

We recall that any compact hyperbolic surface is isometric to a quotient of the hyperbolic plane $\H$ by a Fuchsian
co-compact group $\Gamma \subset \mathrm{PSL}_2(\R)$ -- see \cite{katok1992} for more details.

We could prove \cref{theo:loops} using the same method as we used for \cref{theo:neigh_cyl} and \cref{coro: intersection
  of simple geodesics}, expanding a cylinder around $\delta_z$. However, our initial proof used a different method,
which we decided to present here, in order to expose different ways to use the tangle-free hypothesis.

\begin{proof}[Proof of \cref{theo:loops}]
  \begin{figure}[h]
    \centering
    \begin{subfigure}[b]{0.5\textwidth}
      \centering
      \includegraphics[scale=0.4]{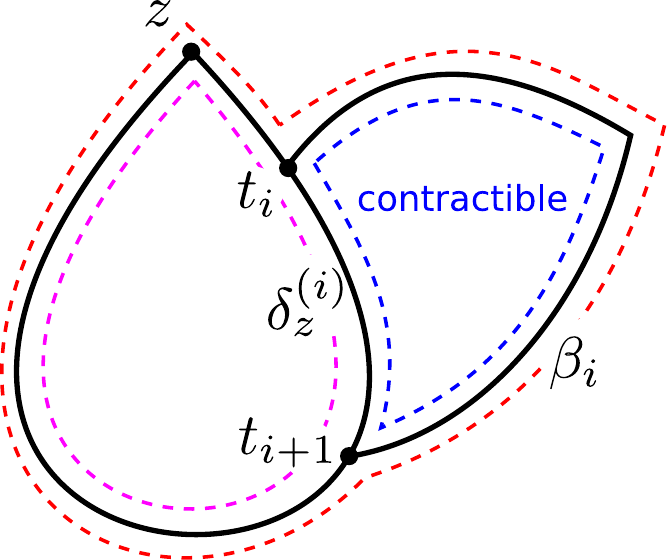}
      \caption{Case $k=0$.}
      \label{fig:proof_theo_loops_simple}
    \end{subfigure}%
    \begin{subfigure}[b]{0.5\textwidth}
      \centering
      \includegraphics[scale=0.4]{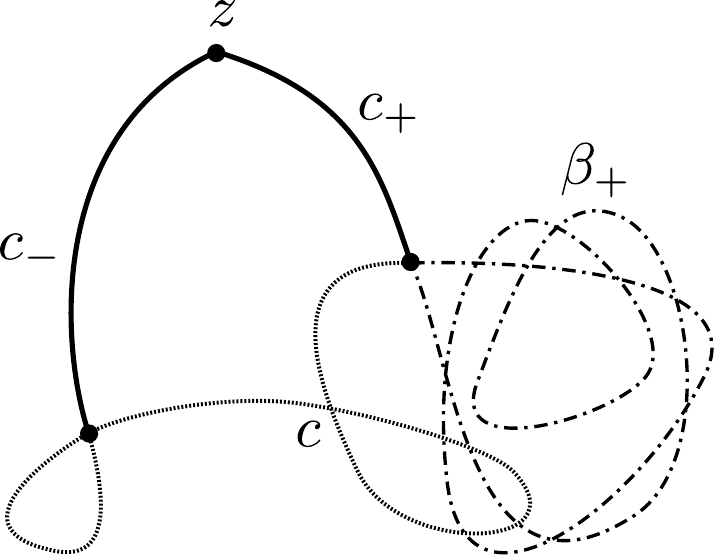}
      \caption{Case $k>0$.}
      \label{fig:proof_theo_loops_k}
    \end{subfigure}%
    \caption{Illustrations of the proof of \cref{theo:loops}.}
    \label{fig:proof_theo_loops}
  \end{figure}
  By replacing $\beta$ by a new curve in its homotopy class, we can assume that $\beta$ has a finite number of
  self-intersections, and of intersections with $\delta_z$, while still satisfying the length condition.
  
  We now prove this result by induction on the number of self-intersections $k\geq 0$ of $\beta$.  We start with the
  base case of $k=0$ so that $\beta$ is simple. We parametrise $\beta : [0,1] \rightarrow X$.  Let
  $0 = t_0 < t_1 < \ldots < t_I = 1$ be the times when $\beta$ meets $\delta_z$.
  
  Let $0 \leq i < I$, and $\beta_i$ be the restriction of $\beta$ to $[t_{i},t_{i+1}]$ -- see 
  \cref{fig:proof_theo_loops_simple}.  Then, the regular neighbourhood $\mathcal{R}$ of $\delta_z$ and $\beta_i$ has
  Euler characteristic $-1$, and total boundary length $\leq 2 (\ell(\delta_z) + \ell(\beta_i)) < 2 L$.  If
  $\mathcal{R}$ is a topological one-holed torus, then by the tangle-free hypothesis, its boundary component is contractible,
  which is impossible for there is no hyperbolic surface of signature $(1,0)$.
  
  Therefore, $\mathcal{R}$ is a topological pair of pants. By the tangle-free hypothesis, one of its boundary
  components is contractible.  It can not be the component corresponding to $\delta_z$, so it is another one. Hence,
  $\beta_i$ is homotopic with fixed endpoints to a portion $\delta_z^{(i)}$ of $\delta_z$.
  
  As a consequence, $\beta = \beta_0 \ldots \beta_{I-1}$ is
  homotopic with fixed endpoints to the product
  \begin{equation*}
    c = \delta_z^{(0)} \delta_z^{(1)} \ldots \delta_z^{(I-1)}.
  \end{equation*}
  $c$ goes from $z$ to $z$ following only portions of $\delta_z$. Therefore, $c$ is homotopic with fixed endpoints to a
  power $\delta_z^j$ of $\delta_z$.
  
  We now move forward to the case $k>0$. We assume the result to hold for any smaller $k$. The idea is to find a way to
  cut $\beta$ into smaller loops on which to apply the induction hypothesis; the construction is represented in 
  \cref{fig:proof_theo_loops_k}.
  
  Let $\ell = \ell(\beta)$. We pick a length parametrisation of $\beta : \R \diagup \ell \mathbb{Z} \rightarrow X$ such
  that $\beta(0) = z$. We look for the first intersection point of $\beta$, starting a $0$, but looking in both
  directions:
  \begin{align*}
    \ell_+ & = \min \{ t \geq 0 \, : \, \exists s \in (t,\ell) \text{ such that } \beta(s) = \beta(t) \} \\
    \ell_- & = \min \{ t \geq 0 \, : \, \exists s \in (t,\ell) \text{ such that } \beta(-s) = \beta(-t) \}.
  \end{align*}
  Up to a change of orientation of $\beta$, we can assume that $\ell_+ \leq \ell_-$. Then, we set
  \begin{equation*}
    t = \max \{ s \in (\ell_+,\ell) \, : \, \beta(s) = \beta(\ell_+) \}
  \end{equation*}
  to be the last time at which $\beta$ visits $\beta(\ell_+)$, so that the restriction of $\beta$ to $[\ell_+,t]$ is a
  loop $\beta_+$. The curve has no self-intersection between $\ell-\ell_-$ and $\ell$, so $t \leq \ell-\ell_-$.  Then,
  if we denote by $c_+$, $c$ and $c_-$ the respective restrictions of $\beta$ to $[0,\ell_+]$, $[t,\ell - \ell_-]$ and
  $[\ell-\ell_-,\ell]$, we can write $\beta = c_+ \, \beta_+ \, c \, c_- $, which is homotopic with fixed endpoints to
  $(c_+ \, \beta_+ \, c_+^{-1}) \, (c_+ \, c \, c_-)$.
  
  Let us apply the induction hypothesis to the two loops $c_+ \, \beta_+ \, c_+^{-1}$ and $c_+ \, c \, c_-$. It will
  follow that they, and hence $\beta$, are homotopic with fixed endpoints to a power of $\delta_z$.
  
  $\beta_+$ is a sub-loop of $\beta$. As a consequence, $c_+ \, c \, c_-$ has less self-intersections than $\beta$, and
  hence strictly less than $k$. Furthermore, it is shorter, so it satisfies the length hypothesis
  $\ell(c_+ \, c \, c_-) + \ell(\delta_z) < L$. So we can apply the induction hypothesis.
  
  $c_+$ is simple and does not intersect $\beta_+$ (except at its endpoint). As a consequence, we can find a curve $b$
  homotopic to $c_+ \, \beta_+ \, c_+^{-1}$ with as many self-intersections as $\beta_+$. $\beta_+$ is a strict sub-loop
  of $\beta$, so this intersection number is strictly smaller than $k$. The length of $b$ can be taken as close as
  desired to that of $c_+ \, \beta_+ \, c_+^{-1}$. Moreover,
  \begin{equation*}
    \ell(c_+ \, \beta_+ \, c_+^{-1})
    = 2 \ell_+ + \ell(\beta_+)
    \leq \ell_+ + \ell_- + \ell(\beta_+)
    \leq \ell(\beta)
  \end{equation*}
  so $b$ can be chosen to satisfy the length hypothesis $\ell(\delta_z) + \ell(b) < L$, and we can apply the induction
  hypothesis to it.
\end{proof}

\subsection{Neighbourhood of a point and graph definition}
\label{sec:neigbourhood-point}

Now that we know about short loops based at a point, we can understand the geometry (and topology) of balls on a
tangle-free surface.

\begin{prop}
  \label{prop:neigh_ball_cyl}
  Let $L>0$, and $X$ be a $L$-tangle-free hyperbolic surface. For a point $z$ in $X$, let
  $\mathcal{B}_{\frac L 8}(z) := \left\{ w \in X \, : \, \dist_X (z,w) < \frac L 8 \right\}$.  Then,
  $\mathcal{B}_{\frac L 8}(z)$ is isometric to a ball in either the hyperbolic plane (whenever the injectivity radius at
  $z$ is $\geq \frac L 8$) or a hyperbolic cylinder.
\end{prop}

In the second case, since the injectivity radius at $z$ is greater than $\frac L 8$, the ball
$\mathcal{B}_{\frac L 8}(z)$ is not contractible on $X$; it is therefore homeomorphic to a cylinder (see
\cref{fig:ball_cylinder}).

\begin{figure}[h]
  \centering
  \includegraphics[scale=0.4]{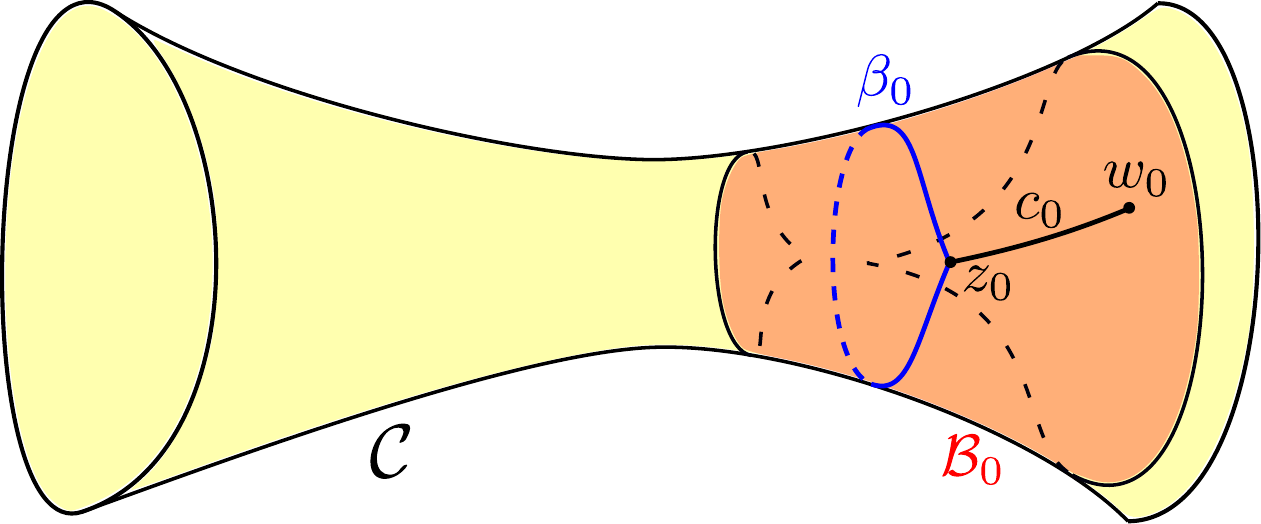}
  \caption{Illustration of the proof of \cref{prop:neigh_ball_cyl} in the cylinder~$\mathcal{C}$. Neighbourhoods
    of points of small injectivity radius on a tangle-free surface are isometric to balls in cylinders, like
    $\mathcal{B}_0$.}
  \label{fig:ball_cylinder}
\end{figure}

In a sense, this corollary proves that our notion of tangle-free implies the natural translation of the notion of
tangle-free for graphs. Indeed, the ball $\mathcal{B}_{\frac L 8}(z)$ has either no non-contractible geodesic loop, or
only one (and its iterates). We could have picked \cref{prop:neigh_ball_cyl} to be a definition for
tangle-free, but we consider the pair of pants definition to be both convenient to use and natural in the context of
hyperbolic geometry and the Weil-Petersson model.

\begin{proof}
  In order to prove this result, we will work in the universal cover $\H$ of~$X$. Let us write
  $X = \faktor{\H}{\Gamma}$, for a co-compact Fuchsian group $\Gamma$.

  Let $z$ be a point on $X$ of injectivity radius smaller than $\frac L 8$ (otherwise, the conclusion is immediate).
  Then, the shortest geodesic loop $\beta$ based at $z$ satisfies $\ell(\beta) < \frac L 4$.

  Let $\tilde z \in \H$ be a lift of $z$, $\tilde{\beta}$ be a lift of $\beta$ starting at $\tilde z$, and
  $\tilde{\mathcal{B}}$ be the ball of radius $\frac L 8$ around $\tilde z$ in $\H$. Let $T_\beta \in \Gamma$ be the
  covering transformation corresponding to $\beta$.  The quotient $\mathcal{C} = \faktor{\H}{\subgp{T_\beta}}$ is a
  hyperbolic cylinder. The ball $\tilde{\mathcal{B}}$ is projected on a ball $\mathcal{B}_0$ on $\mathcal{C}$. Let us
  prove that the projection from $\mathcal{B}_0$ on $\mathcal{C}$ to $\mathcal{B}$ on $X$ is an isometry.

  In order to do so, we shall establish that for any $\tilde{w} \in \tilde{\mathcal{B}}$, the set of transformations
  $T \in \Gamma$ such that $T \cdot \tilde w \in \tilde{\mathcal{B}}$ is included in $\subgp{T_\beta}$. Since any two
  points in $\tilde{\mathcal{B}}$ are at a distance at most $\frac L 4 < \frac L 2$, this will follow from proving
  \begin{equation*}
    \Gamma_L(\tilde{w}) := \left\{ T \in \Gamma \; : \; \dist_{\H}(\tilde{w}, T \cdot \tilde{w}) < \frac L 2 \right\} \subset \subgp{T_\beta}.
  \end{equation*}
  Let $c$ be the shortest path from $\tilde{w}$ to $\tilde z$. The path $c \; \tilde{\beta} \; (T_\beta \circ c^{-1})$ is a path
  from $\tilde{w}$ to $T_\beta \cdot \tilde{w}$. Its length is $2 \ell(c) + \ell(\beta) < 2 \times \frac L 8 + \frac L 4 = \frac L
  2$. As a consequence, $T_\beta$ belongs in $\Gamma_L(\tilde{w})$. Then, $\Gamma_L(\tilde{w})$ is not reduced to $\{ \id \}$.  By
  \Cref{cor:close_Gamma}, it is included in a cyclic subgroup $\subgp{T_0}$. $T_\beta$ hence is a power of
  $T_0$, but $T_\beta$ is primitive. Therefore, $T_\beta = T_0^{\pm 1}$, and the conclusion follows.
\end{proof}

\subsection{Short geodesics are simple}
\label{sec:short-geod}

 \begin{coro}
   \label{coro:simple}
   Let $L>0$, and $X$ be a $L$-tangle-free hyperbolic surface. Any primitive closed geodesic on $X$ of length $< L$ is
   simple.
 \end{coro}

 This consequence of Theorem \ref{theo:loops} can also be deduced from the fact that the shortest non-simple primitive
 closed geodesic on a compact hyperbolic surface is a figure eight geodesic \cite[Theorem 4.2.4]{buser1992}, which is
 embedded in a pair of pants or one-holed torus.

\begin{proof}
  Let us assume by contradiction that $\gamma$ is not simple; we can then pick an intersection point $z$. This allows us
  to write $\gamma$ as the product of two geodesic loops $\gamma_1$, $\gamma_2$ based at $z$. Since
  $\ell(\gamma_1) + \ell(\gamma_2) < L$, one of them is $<L/2$. Up to a change of notation, we take it to be $\gamma_1$.

  Let $\delta_z$ be the shortest geodesic loop based at $z$. By definition, $\ell(\delta_z) \leq \ell(\gamma_1)$. So
  $\gamma_1$ and $\gamma_2$ both satisfy the length hypothesis of \cref{theo:loops}:
  \begin{align*}
    \ell(\gamma_1) + \ell(\delta_z) & \leq 2 \ell(\gamma_1) < L \\
    \ell(\gamma_2) + \ell(\delta_z) & \leq \ell(\gamma) < L.
  \end{align*}
  Therefore, they are both homotopic with fixed endpoints to powers of $\delta_z$, which implies $\gamma$ is too. So
  $\gamma$ is freely homotopic to a power $j$ of the simple closed geodesic $\gamma_0$ in the free homotopy class of
  $\delta_z$. By uniqueness, $\gamma = \gamma_0^j$. $\gamma$ is primitive, so $j = 0$ or $1$. But $\gamma$ is not
  contractible (so $j \neq 0$) and not simple (so $j \neq 1$): we reach a contradiction, which allows us to conclude.
\end{proof}

\begin{rema}
  Put together, \cref{coro:simple} and \ref{coro: intersection of simple geodesics} imply that all primitive closed
  geodesics of length $< \frac L 2$ are simple and disjoint.  Any such family of curves has cardinality at most
  $2g-2$. But we know that the number of primitive closed geodesics of length $< \frac L 2$ on a fixed surface is
  equivalent to $\frac{2}{L} e^{\frac L 2}$ as $L \rightarrow + \infty$ \cite{huber1959,buser1992}. This can be seen as
  another indicator that, if $X$ is $L$-tangle-free of large genus, then we expect $L$ to be at most logarithmic in $g$.
\end{rema}

\section{Any surface of genus $g$ is ($4 \log g + O(1)$)-tangled}
\label{sec:upper_bound_L}

We recall that any surface is $L$-tangled for $L = \frac{3}{2}\mathcal{B}_g$, the Bers constant, because it can be
entirely decomposed in pairs of pants of maximal boundary length smaller than $\mathcal{B}_g$. The best known estimates
on the Bers constant $\mathcal{B}_g$ are linear in the genus~$g$ \cite{buser1992a,parlier2014}, which is pretty far off
the $c \log g$ we obtained for random surfaces.  This is not a surprise, because in order to prove that a surface is
tangled, we only need to find one embedded pair of pants or one-holed torus. In Buser and Parlier's estimates on~$\mathcal{B}_g$
\cite{buser1992,parlier2014}, the pair of pants decomposition is constructed by successively exhibiting short curves on
the surface; the first ones are of length $\simeq \log g$, but as the construction goes on, and we find $2g-2$ curves to
entirely cut the surface, a linear factor appears.

In our case, we only need to stop the construction as soon as we manage to separate a pair of pants.  Following Parlier's
approach in \cite{parlier2014} to bound the Bers constant, we prove the following.

\begin{prop}
  \label{prop:upper_bound}
  There exists a constant $C>0$ such that, for any $g \geq 2$, any compact hyperbolic surface of genus $g$ is $X$ is
  $L$-tangled for $L = 4 \log g + C$.
\end{prop}

This goes to prove that random hyperbolic surfaces are almost optimally tangle-free, despite the possibility of having a
small injectivity radius.

The proof relies on the following two Lemmas, which are all used by Parlier~\cite{parlier2014}. 
\cref{lemm:bavard}, due to Bavard \cite{bavard1996}, allows us to find a small geodesic loop on our surface.

\begin{lemm}
  \label{lemm:bavard}
  Let $X$ be a hyperbolic surface of genus $g$. 
  For any $z \in X$, the length of the shortest geodesic loop through $z$ is smaller than
  \begin{equation*}
    2 \arccosh \paren*{\frac{1}{2 \sin \frac{\pi}{12g-6}}} = 2 \log g + O(1).
  \end{equation*}
\end{lemm}

Some problems will arise in the proof if the geodesic loop we obtain using this result is too small. These difficulties
can be solved by assuming a lower bound on the injectivity radius of the surface; for instance, for random surfaces,
with high probability, one can assume that the injectivity radius is bounded below by $g^{-\epsilon}$ for a
$\epsilon > 0$ \cite{mirzakhani2013}. However, such an assumption makes the final inequality weaker.

Another way to fix this issue, used in \cite{parlier2014}, is to expand all the small geodesics, and by this process
obtain a new surface, with an injectivity radius bounded below, and in which the lengths of all the curves are
longer. For our purposes, we only need to expand one curve. This is achieved by the following Lemma.

\begin{lemm}[Theorem 3.2 in \cite{parlier2005}]
  Let $S_{g,n}$ be a base surface with $n>0$ boundary components.  Let
  $(X,f) \in \mathcal{T}_{g,n}(\ell_1, \ldots, \ell_n)$ and $\epsilon_1, \ldots, \epsilon_n \geq 0$.  Then, there exists
  a marked hyperbolic surface $(\tilde X, \tilde f)$ in
  $\mathcal{T}_{g,n}(\ell_1+\epsilon_1, \ldots, \ell_n + \epsilon_n)$ such that, for any closed curve $c$ on the base
  surface $S_{g,n}$, $\ell_X(c) \leq \ell_{\tilde X}(c)$.
\end{lemm}

We are now able to prove the result.

\begin{proof}
  Let $\gamma$ be the systole of $X$ which is necessarily simple.  We cut the surface $X$ along this curve, and obtain a
  (possibly disconnected) hyperbolic surface $X_{\mathrm{cut}}$ with two boundary components. By the extension Lemma
  (applied to both components separately if need be), there exists a surface $X_{\mathrm{cut}}^+$ such that:
  \begin{itemize}
  \item the boundary components $\beta_1$, $\beta_2$ in $X_{\mathrm{cut}}^+$ are of length
    $ 1 \leq \ell \leq 2 \log g + O(1)$.
  \item for any closed curve $c$ not intersecting $\gamma$,
    $\ell_{X_{\mathrm{cut}}}(c) \leq \ell_{X_{\mathrm{cut}}^+}(c)$.
  \end{itemize}
  We shall find a pair of pants in $X_{\mathrm{cut}}^+$, and use the relationship between lengths in $X$ and
  $X_{\mathrm{cut}}^+$ to conclude.

  For $w > 0$, let us consider the $w$-neighbourhood of one component $\beta_1$ of the boundary of $X_{\mathrm{cut}}^+$
  \begin{equation*}
    \mathcal{C}_w (\beta_1) = \{ z \in X_{\mathrm{cut}}^+ \, : \, \dist(z, \beta_1) < w \}.
  \end{equation*}
  For small enough $w$, $\mathcal{C}_w(\beta_1)$ is a half-cylinder. However, there is a $w$ at which this isometry
  stops. This $w$ can be bounded by a volume argument: as long as $\mathcal{C}_w$ is a half-cylinder,
  \begin{equation*}
    \Vol (\mathcal{C}_w(\beta_1)) = \ell \sinh w \leq \Vol X = 2 \pi(2g-2).
  \end{equation*}
  However, $\ell \geq 1$. It follows that $w \leq \log g + O(1)$.

  There are two reasons for this isometry to stop.
  \begin{itemize}
  \item The half-cylinder self-intersects inside the surface (see \cref{fig:cylinder_reg_neig_1}). Then, one can
    construct an embedded pair of pant on $X_{\mathrm{cut}}^+$, of total boundary length $\leq 2\ell + 4 w$. This pair
    of pant will also be one on $X$, with shorter boundary components.
  \item The half-cylinder reaches the boundary of $X_{\mathrm{cut}}^+$. It can only do so by intersecting the component
    $\beta_2$. Then, one can construct an embedded pair of pant on $X_{\mathrm{cut}}^+$ of boundaries shorter than $\ell$,
    $\ell$, and $2 \ell + 2 w$, which corresponds to a one-holed torus on $X$, of boundary shorter than $2 \ell + 2 w$ (see
    \cref{fig:cylinder_reg_neig_2}, but expanding only a half-cylinder).
  \end{itemize}

  We can conclude that the surface $X$ is $L$-tangled, for $L =  \ell +  2w \leq 4 \log g + O(1)$.
\end{proof}

\bibliographystyle{plain}
\bibliography{bibliography}

\end{document}